\documentclass[11pt]{amsart}

\usepackage[english]{babel}
\usepackage{graphicx}
\usepackage{framed}
\usepackage[normalem]{ulem}
\usepackage{amsmath}
\usepackage{amsthm}
\usepackage{amssymb}
\usepackage{amsfonts}
\usepackage{enumerate}
\usepackage{hyperref}
\usepackage{xcolor}
\usepackage{mathtools}
\usepackage{fullpage}
\usepackage{cases}
\usepackage{graphics}
\usepackage{tikz-cd}
\usepackage{caption, subcaption}
\usepackage{harpoon}
\usepackage[utf8]{inputenc}
\usepackage[top=1 in,bottom=1in, left=1 in, right=1 in]{geometry}
\usepackage{cleveref}

\newcommand{\RR}{\mathbb{R}}

\newcommand{\QQ}{\mathbb{Q}}
\newcommand{\ZZ}{\mathbb{Z}}
\newcommand{\CC}{\mathbb{C}}

\renewcommand{\phi}{\varphi}

\newcommand*{\vect}[1]{\overrightharp{\ensuremath{#1}}}

\DeclareMathOperator{\mult}{mult}
\DeclareMathOperator{\Cone}{Cone}
\DeclareMathOperator{\Hom}{Hom}
\DeclareMathOperator{\Span}{Span}

\DeclareMathOperator{\prim}{prim}

\newcommand{\introthmname}{}
\newtheorem{introthminn}{\introthmname}
\newenvironment{introthm}[1]
  {\renewcommand{\introthmname}{#1}\begin{introthminn}}
  {\end{introthminn}}

\theoremstyle{definition}

\newtheorem{corollary}{Corollary}[section]
\newtheorem{definition}{Definition}[section]
\newtheorem{construction}{Construction}[section]

\newtheorem{example}{Example}[section]
\newtheorem{proposition}{Proposition}[section]

\newtheorem{remark}{Remark}[section]

\numberwithin{equation}{section}

\definecolor{DarkRed}{RGB}{200,20,20}
\definecolor{DarkGreen}{RGB}{0,120,0}
\definecolor{SkyBlue}{rgb}{0.16, 0.32, 0.75}
\hypersetup{
    colorlinks=true,
linkcolor = blue,    linkbordercolor=DarkRed,
     citecolor = teal,
    }

\title{Rational tensegrities through the lens of toric geometry}
\author{Fatemeh Mohammadi and Xian Wu}
\date{}

\begin{document}
\begin{abstract}
A classical tensegrity model consists of an embedded graph in a vector space with rigid bars representing edges, and an assignment of a stress to every edge such that at every vertex of the graph the stresses sum up to zero. 
The tensegrity frameworks have been recently extended from the two
dimensional graph case to the multidimensional setting. We study the multidimensional tensegrities using tools from toric geometry. For a given rational tensegrity framework $\mathcal{F}$, we construct a glued toric surface $X_\mathcal{F}$. We show that the abelian group of tensegrities on $\mathcal{F}$ is isomorphic to a subgroup of the Chow group $A^1(X_\mathcal{F};\QQ)$. In the case of planar frameworks, we show how to explicitly carry out the computation of tensegrities via classical tools in toric geometry.
\end{abstract}
\maketitle

\vspace{-5mm}

\section{Introduction}
This work is concerned with the development of new connections between the multidimensional tensegrity frameworks, toric varieties, and their Chow groups.

\subsection{Tensegrity frameworks.}\label{higher}
A classical tensegrity model \cite{maxwell1864xlv} consists of an embedded graph in a vector space $\RR^d$ with rigid bars as edges, and a balancing condition at each vertex, which gives a stable structure (see, e.g., \cite{roth1981tensegrity, connelly1996second,connelly2013tensegrity}). Tensegrities have a wide range of applications in different areas of modern science and engineering technology (see, e.g., \cite{motro2003tensegrity, roth1981tensegrity, juan2008tensegrity, zhang2015tensegrity}). 
The notion of tensegrity has also been developed in higher dimensions (see, e.g., \cite{karpenkov2021geometric,karpenkov2022equilibrium,rybnikov1999stresses, rybnikov2000polyhedral}). Some theories about the existence of tensigrities and stratifications are recently developed in \cite{karpenkov2021frame,doray2010geometry}.  In this paper, we focus on the multidimensional tensegrities introduced in \cite{karpenkov2022equilibrium}, and we examine their structures over $\QQ$ or $\ZZ$ from the algebraic geometry perspective.

\smallskip
We now define the main object of this paper, the multidimensional tensegrity framework (see, \cite{karpenkov2022equilibrium} for more details).
    Let $N\cong\ZZ^d$ be a lattice and $N_\RR= N\otimes_\ZZ\RR\cong\RR^d$. A \emph{$k$-framework} $\mathcal{F}=(E,F,I,{\bf n})$ in $N_\RR$ consists of the following data:
    \begin{itemize}
        \item a collection $E$ of $(k-1)$-dimensional affine subspaces in $\RR^d$; 
        \item a collection $F$ of $k$-dimensional affine subspaces in $\RR^d$;
        \item a subset $I\subset\{(p,q)\in E\times F\mid\  p\subset q\}$;
        \item a function $\bf{n}$ assigning to each pair $(e,f)$ in $I$, a vector ${\bf n}(e,f)$ in $f$ normal to $e$, which is mapped to the primitive generator of the lattice $N/(N\cap e)$ under $\pi_e:N\rightarrow N/(N\cap e)$. \\
        See \Cref{framework}.
    \end{itemize} 
The elements of $E$, $F$ and $I$ are called \emph{edges}, \emph{faces}, and \emph{incidences}, respectively. 
A \emph{stress} $w$ on $\mathcal{F}$ is a function $w:F\rightarrow\QQ$. In particular, $w$ is called a \emph{self-stress} if for every $e\in E$, we have:
\begin{equation}\label{selfstress}
    \sum_{(e,f)\in I}{\bf n}(e,f)w(f)=0. 
    \end{equation}
The collection of self-stresses forms an abelian group $A_\mathcal{F}$ under addition. Moreover, $\mathcal{F}$ is called a \emph{tensegrity} if there exists a nonzero self-stress on it.

\subsection{Toric varieties and Chow groups.}\label{sec:Chow}
A normal algebraic variety $X$ is \emph{toric} if there exists a $(\CC^*)^n$-action on $X$ with an open dense orbit isomorphic to the torus $T=(\CC^*)^n$. Toric varieties form an important family of varieties in algebraic geometry, mainly because they are linked to the theory of lattices, polytopes and polyhedral fans. Moreover, their geometric properties are encoded as combinatorial invariants of their corresponding polytopes. In some sense, toric varieties are the easiest objects to deal with in algebraic geometry, and they can be used in the study of arbitrary varieties via degeneration techniques \cite{anderson2013okounkov,alexeev2002complete, bossinger2021families}. 
The standard references for toric geometry are \cite{oda1983convex,fulton1993introduction} and \cite{cox2011toric}.

\smallskip
To define the Chow group, consider the lattice of cocharacters $N=\Hom(\CC^*,T)\cong\ZZ^n$ and the lattice of characters $M=\Hom(N,\ZZ)$. Any complete fan $\Sigma$ in $N_\RR=N\otimes_\ZZ\RR$ uniquely determines a toric variety $X_\Sigma$. The Chow group of $k$-dimensional algebraic cycles, denoted by $A_k(X_\Sigma;\QQ)$ is described in \cite[Proposition~1.1]{fulton1997intersection}. Let $\Sigma^k$ be the set of cones in $\Sigma$ of codimension $k$. The $T$-invariant closed subvariety associated to $\sigma\in\Sigma^k$ is denoted by $V(\sigma)$. Then $A_k(X_\Sigma;\QQ)$ is generated by the rational equivalent classes $[V(\sigma)]$ where $\sigma$ runs over $\Sigma^k$, and the relations are given by
\begin{equation}\label{rel}
    \sum_{\sigma\in\Sigma^k,\ \sigma\supset\tau}\langle m,n_{\sigma,\tau}\rangle[V(\sigma)]=0,
\end{equation}
for every $\tau\in\Sigma^{k+1}$ and all $m\in M(\tau)=\tau^\perp\cap M$. Here $n_{\sigma,\tau}$ is a lattice point in $\sigma$ whose image generates the 1-dimensional lattice $N_\sigma/N_\tau$, where $N_\sigma$ and $N_\tau$ are sublattices of $N$ generated by $N\cap\sigma$ and $N\cap\tau$, respectively.

\smallskip

On the dual side, the \emph{operational Chow cohomology ring} is defined as  $A^\bullet(X_\Sigma;\QQ)=\oplus_k A^k(X_\Sigma;\QQ)$ (see \cite[Chapter 17]{fulton2013intersection}). When $X_\Sigma$ is $\QQ$-factorial, or equivalently, every cone in $\Sigma$ is simplicial, one can identify $A^k(X_\Sigma;\QQ)$ with $A_{\dim(X_\Sigma)-k}(X_\Sigma;\QQ)$. By \cite[Theorem 3]{fulton1994intersection}, there is an isomorphism
\begin{equation}\label{dual}
    A^k(X_\Sigma;\QQ)\cong\Hom_\QQ(A_k(X_\Sigma;\QQ),\QQ).
\end{equation}
Moreover, the Chow cohomology group $A^k(X_\Sigma;\QQ)$ is isomorphic to the group of Minkowski weights on $\Sigma^k$. See \cite[Theorem~2.1]{fulton1997intersection}. 
In particular, a $\QQ$-valued function $c$ on $\Sigma^k$ is a \emph{Minkowski weight} if it satisfies the following balancing condition:
\begin{equation}\label{Minkowski}
        \sum_{\sigma\in\Sigma^k,\ \sigma\supset\tau}\langle m,n_{\sigma,\tau}\rangle c(\sigma)=0,
    \end{equation}
for every $\tau\in\Sigma^{k+1}$ and $m$ in the lattice $M(\tau)$.

\subsection{Outline and our results.} 
In \Cref{higherandtoric}, we first construct a toric variety for every edge in $\mathcal{F}$, and then glue their associated polytopes along a proper choice of faces to obtain a polytope for $\mathcal{F}$. Then we use this polytope to construct a toric variety $X_{\mathcal{F}}$ associated to $\mathcal{F}$ (see Construction~\ref{const} and \Cref{assembled}).
Our main goal is to prove the following theorem which relates the Chow group of $X_{\mathcal{F}}$ from \Cref{sec:Chow}, and the abelian group of self-stresses $A_{\mathcal{F}}$ from \Cref{higher}.

\begin{introthm}{Theorem}\label{thm:main}
Consider a framework $\mathcal{F}$, and let $A_\mathcal{F}$ be the group of self-stresses on $\mathcal{F}$. Let $X_{\mathcal{F}}$ be a glued toric variety from Construction~\ref{const} and \Cref{assembled}, and let $A^1(X_{\mathcal{F}};\QQ)_\mathcal{F}$ be the subgroup of the Chow group $A^1(X_{\mathcal{F}};\QQ)$ with cocycles vanishing on the reference rays. 
Then we have that:
    $$
        A^1(X_{\mathcal{F}};\QQ)_\mathcal{F}\cong A_\mathcal{F}.
    $$
\end{introthm}
In \Cref{planar}, we focus on the classical tensegrity model of planar graphs. In \Cref{sec:graph}, we construct a polyhedral fan for any planar graph, and equip that with an irreducible toric variety. This enables us to explicitly compute the corresponding Chow rings in \Cref{sec:computation}, and so tensegrities, using Stanley-Reisner ideals. We also provide a computational example (see Example~\ref{ex:running}).

\smallskip
\noindent{\bf Acknowledgement.}
The first author would like to express her gratitude to the organizers of the Fields Institute Thematic Program on Geometric Constraint Systems, Framework Rigidity, and Distance Geometry, for introducing her to the subject, and for many helpful conversations. She would also like to thank James Cruickshank, Anthony Nixon, and Shin-ichi Tanigawa for helpful discussions during the project \cite{cruickshank2022global}.
The authors would like to thank Oleg Karpenkov for helpful discussions. The authors were partially supported by the 
FWO grants G0F5921N (Odysseus programme), G023721N, and BOF/STA/201909/038. 

\section{Glued toric varieties associated to multidimensional frameworks}\label{higherandtoric}
Throughout we fix a multidimensional framework $\mathcal{F}$ as defined in \Cref{higher}, which is the integral version of \cite[Definition~2.1]{karpenkov2022equilibrium}. We assume that $\mathcal{F}$ is a \emph{rational} framework, i.e.~each element in $E$ or $F$ contains infinitely many rational points. 
We also assume that $\mathcal{F}$ is \emph{generic}, i.e.~all $f$'s
with $(e, f)\in I$ are distinct for a fixed $e\in E$. Moreover, for any $e\in E$, we have that $\#\{f\mid(e,f)\in I\}\geqslant 3$ if it is nonzero. 

\medskip

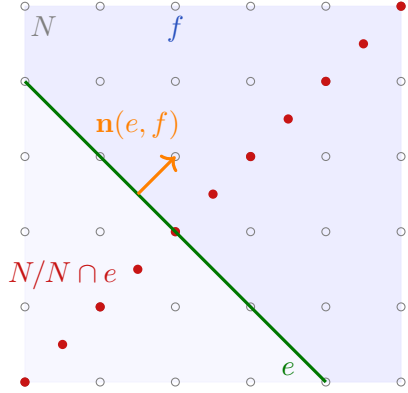
\begin{figure}[h]
    \centering
    \begin{tikzpicture}[scale=0.5]
    \filldraw[blue!10, opacity=0.7] (8,0) -- (10,0) -- (10,10) -- (0,10) -- (0,8) -- cycle;
    \filldraw[blue!10, opacity=0.3] (8,0) -- (0,0) -- (0,8) --  cycle;
    \foreach \i in {0,2,...,10}
      \foreach \j in {0,2,...,10}{
        \draw [gray] (\i,\j) circle(3pt);
        };
     \foreach \i in {0,...,10}
      \foreach \j in {0,...,10}{
        \ifnum \i = \j
            \fill[DarkRed] (\i,\j) circle(3.5pt);
        \fi
      };
     \draw [very thick, DarkGreen] (8,0) -- (0,8);
     \draw [very thick, orange, ->] (3,5) -- (4,6);
     \node [below] at (4,10) {\color{SkyBlue}$f$};
     \node [below] at (7,0.8) {\color{DarkGreen}$e$};
     \node [below] at (1,3.5) {\color{DarkRed}$N/N\cap e$};
     \node [below] at (0.5,10) {\color{gray}$N$};
     \node [below] at (3,7.5) {\color{orange}${\bf n}(e,f)$};
  \end{tikzpicture}
  \caption{Local picture of a framework.}
    \label{framework}
\end{figure}

We now explain our method to associated a toric variety to any given framework $\mathcal{F}$. 

\begin{construction}[The toric variety $X_{\mathcal{F}}$]\label{const} We first construct a toric variety for every edge in $\mathcal{F}$, and then glue their associated polygons to obtain a polytope, and hence a toric variety for $\mathcal{F}$.
\smallskip

More precisely, for each $e$, locally, in a neighborhood $U_e=e\times\Delta$, where $\Delta$ is a small ball of dimension $(n-k+1)$, we contract $U_e$ along $e$. Extending rays, we obtain a polyhedral fan $\Sigma_e'\subset\RR^2$ centered at the image of $e$ under the contraction, whose rays are contractions of $\{f\mid(e,f)\in I\}$, and dimension 2 cones are naturally cut by rays. Then, we complete the fan, if necessary, as follows. If all the rays $\rho_1,\ldots,\rho_s$ lie on the same half-plane, then we add a new ray $\rho_0$ generated by $-\sum_{i=1}^s v_i$, where $v_i$ is the primitive generator of $\rho_i$. Including the two $2$-dimensional cones containing $\rho_0$, we obtain a complete fan $\Sigma_e$. We call the ray $\rho_0$ an \emph{assistant ray}. All these operations are canonical with respect to $\mathcal{F}$. We now proceed with our construction as follows:
\begin{itemize}
    \item Let $X_{\Sigma_e}$ be the complete toric variety associated to $\Sigma_e$. Note that $X_{\Sigma_e}$ is projective since $\dim X_{\Sigma_e}=2$. We choose an ample line bundle $L_e$ on $X_{\Sigma_e}$, equivalently, a polygon $P_e$ whose normal fan is $\Sigma_e$.
    \item If $(e,f)$ and $(e',f)$ are both in $I$, and $l_f,l_{f'}$ are edges in $P_e,P_{e'}$ normal to the contraction of $f$, respectively, then we choose a bijection $\phi_f^{e,e'}$ between $l_f$ and $l_{f'}$, and glue the polygons $P_e,P_{e'}$ via $\phi_f^{e,e'}$ along the edges $l_f,l_{f'}$. We denote $\Phi_F^E$ for the set of all bijections $\phi_f^{e,e'}$.
\end{itemize}
We call the pair $\mathcal{P}:=(\cup_{e\in E}P_e,\Phi_F^E)$  the \emph{glued polygon}. We associate a variety $X$ to $\mathcal{P}$, where each component $X_{\Sigma_e}$ and the gluing is given by $\Phi_F^E$. Note that, the two non-canonical steps in the construction may lead to  multiple varieties $X$ (associated to $\mathcal{F}$), however by \Cref{rel}, the Chow groups are the same. Hence, we write $A^1(X_\mathcal{F};\QQ)$ for a choice of $X_\mathcal{F}$ associated to $\mathcal{P}$.
\end{construction}

\begin{example}

\begin{figure}[h]
\begin{tikzpicture}
    \draw [thick, SkyBlue] (1.2,-0.2) -- (1.2,3.8) -- (3.2,1.8) -- cycle;
    \draw [thick, SkyBlue] (1,0) -- (1,4) -- (3,3) -- (2.6,2.4);
    \draw [thick, SkyBlue] (1,0) -- (1.2,0.3);
    \draw [thick, SkyBlue, dashed] (2.6,2.4) -- (1.2,0.3);
    \draw [thick, SkyBlue] (-1.2,-0.2) -- (-1.2,3.8) -- (-3.2,1.8) -- cycle;
    \draw [thick, SkyBlue] (-1,0) -- (-1,4) -- (-3,3) -- (-2.6,2.4);
    \draw [thick, SkyBlue] (-1,0) -- (-1.2,0.3);
    \draw [thick, SkyBlue, dashed] (-2.6,2.4) -- (-1.2,0.3);
    \node [right] at (-1,2.2) {$l_{f_1}$};
    \node [left] at (-1.2,2) {$l_{f_2}$};
    \node [right] at (1.2,2) {$l_{f_3}$};
    \node [left] at (1.1,2.2) {$l_{f_4}$};
\end{tikzpicture}\hspace{1.5cm}
\begin{tikzpicture}
    \draw [thick, SkyBlue] (0,0) -- (0,4) -- (2,2) -- cycle;
    \draw [thick, SkyBlue] (0,0) -- (0,4) -- (2,3) -- (1.6,2.4);
    \draw [thick, SkyBlue, dashed] (1.6,2.4) -- (0,0);
    \draw [thick, SkyBlue] (0,0) -- (0,4) -- (-2,2) -- cycle;
    \draw [thick, SkyBlue] (0,0) -- (0,4) -- (-2,3) -- (-1.6,2.4);
    \draw [thick, SkyBlue, dashed] (-1.6,2.4) -- (0,0);
\end{tikzpicture}
\caption{A realizable glued polygon in $\RR^3$.}
    \label{gluedpoly}
\end{figure}
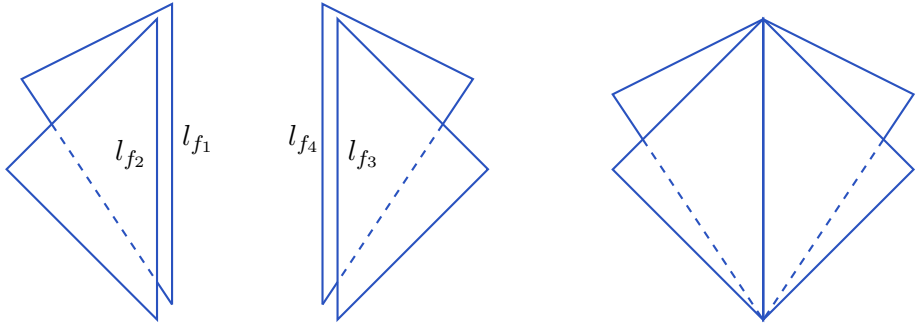
Let $\mathcal{F}$ be the following framework. Fix the points
\begin{displaymath}
  \begin{split}
    B_1 & = (1,1,0),\\
    B_5 & = (2,2,1),\\
  \end{split}
\quad\quad
  \begin{split}
    B_2 & = (-1,1,0),\\
    B_6 & = (-2,2,1),\\
  \end{split}
  \quad\quad
  \begin{split}
B_3 & = (-1,-1,0),\\
B_7 & = (-2,-2,1),\\
\end{split}
\quad\quad
  \begin{split}
    B_4 & = (1,-1,0),\\
    B_8 & = (2,-2,1),\\
  \end{split}\end{displaymath}
\begin{center}
    and $B_i=B_{i-4}-(0,0,2)$ for $i=9,\ldots,12$.
\end{center} 
Denote $B_iB_j$ for the line through any pair of points $B_i$ and $B_j$, and $B_iB_jB_kB_\ell$ for the half-plane containing any collection of four coplanar points $B_i,B_j,B_k,B_\ell$. Set
\begin{displaymath}
  \begin{split}
    e_1 & = B_1B_2,\\
  \end{split}
\quad\quad
  \begin{split}
    e_2 & = B_2B_3,\\
  \end{split}
  \quad\quad
  \begin{split}
e_3 & = B_3B_4,\\
\end{split}
\quad\quad
  \begin{split}
    e_4 & = B_4B_1,\\
  \end{split}\end{displaymath}
\begin{center}
    and $e_i=\begin{cases}
    B_iB_{i-4},i=5,\ldots,8;\\
    B_iB_{i-8},i=9,\ldots,12.
    \end{cases}$ 
\end{center} 
We also set
\begin{displaymath}
  \begin{split}
   f_0 & = B_1B_2B_3B_4,\\
   f_4 & = B_4B_1B_5B_8,\\
  \end{split}
\quad\quad
  \begin{split}
    f_1 & = B_1B_2B_6B_5,\\
    f_5 & = B_1B_2B_{10}B_9,\\
  \end{split}
  \quad\quad
  \begin{split}
f_2 & = B_2B_3B_7B_6,,\\
f_6 & = B_2B_3B_{11}B_{10},\\
\end{split}
\quad\quad
  \begin{split}
 f_3 & = B_3B_4B_8B_7,\\
   f_7 & = B_3B_4B_{12}B_{11},\\
  \end{split}
  \end{displaymath}
  \begin{center}
    and $f_8=B_4B_1B_9B_{12}$.
\end{center} 
 Note that the subset $I$ (from \Cref{{higher}}) consists of all $(e,f)$'s with $e\in\{e_1,\ldots,e_4\}$. 
Moreover, the vectors ${\bf n}(e,f)$'s are pointing towards the interior of the labelled quadrilaterals. \Cref{gluedpoly} shows a choice of the glued polygon realizable in $\RR^3$.
\end{example}

\begin{figure}[h]
    \centering
\begin{tikzpicture}[scale=0.7]
    \filldraw[blue!10, opacity=0.5]  (0,0) -- (-1,-1) -- (1,-1) -- (3,0) -- (3,2) -- (-3,2) -- (-3,0) -- (-1.3,-0.7) -- (-0.7,-0.1) -- (-2,1) -- (2,1) -- cycle;
    \draw [thick, SkyBlue] (0,0) -- (-1,-1) -- (1,-1) -- cycle;
    \draw [thick, SkyBlue] (0,0) -- (2,1) -- (3,0) -- (1,-1);
    \draw [thick, SkyBlue] (2,1) -- (3,2) -- (3,0);
    \draw [thick, SkyBlue] (3,2) -- (-3,2) -- (-3,0) -- (-2,1) -- (-3,2);
    \draw [thick, SkyBlue] (-2,1) -- (2,1);
    \draw [thick, SkyBlue] (-3,0) -- (-1.3,-0.7) -- (-0.7,-0.1) -- (-2,1);
    \draw [line width=1mm, DarkRed] (-1,-0.4) -- (-0.7,-0.7);
    \draw [very thick, DarkRed] (-1,-0.6) -- (-1,-0.4) -- (-0.8,-0.4);
    \draw [very thick, DarkRed] (-0.7,-0.5) -- (-0.7,-0.7) -- (-0.9,-0.7);
\end{tikzpicture}
    \caption{Monodromy.}
    \label{mono}
\end{figure}
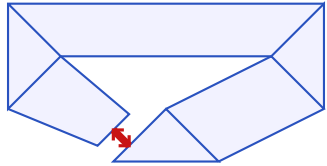

\begin{remark}\label{monodromy}
   We note that due to the existence of monodromy, the glued polygon $\cup_{e\in E}P_e$ and $\Phi_F^E$ might  not be realizable in a vector space. See, e.g., \Cref{mono}.
\end{remark}

\begin{definition}\label{assembled}
    A \emph{glued toric variety} $X$ is a union $X=\cup_i X_i$ of complete toric varieties which are glued along toric invariant subvarieties. A \emph{polarized glued toric variety} $(X,L)$ is a pair of a glued toric variety $X$ and a linearized ample line bundle $L$ on $X$, where $L|_{X_i}$ is also ample for every $i$.
\end{definition}

\begin{remark}
    Note that the resulted toric varieties in Construction~\ref{const} may not be seminormal, and so they are not necessarily the \emph{stable toric varieties (or broken toric varieties)} appeared in \cite{alexeev2002complete,olsson2008compactifying}. Without fixing $\{\phi_f^{e,e'}\}$, one only determines a family of varieties, which may not be isomorphic to each other, see \cite[Section 2.2]{alexeev2015moduli}. 
\end{remark}
We now have all the ingredients to prove our main theorem.
\begin{proof}[{\bf Proof of Theorem~\ref{thm:main}}]
    For each complete fan $\Sigma_e$, let $\rho_f$ be the ray arising from the contraction of $f$. Then $A_1(X_{\mathcal{F}};\QQ)_\mathcal{F}$ is generated by all $f$'s, where the relations are given by \Cref{rel} and $\{\phi_f^{e,e'}\}$. Now, by setting $c(\rho_f)$ to be $w(f)$, we observe that the balancing condition for Minkowski weights \eqref{Minkowski} and for tensegrities are the same, which completes the proof.
\end{proof}


\section{Classical Planar Tensegrities}\label{planar}
In this section, we study the case, where $d=2$ in \Cref{higher} and $F$ has a bounded support. In this case, the framework $\mathcal{F}$ has an underlying graph $G$, which is the classical tensegrity model (see, e.g., \cite{maxwell1864xlv,roth1981tensegrity, connelly1996second,connelly2013tensegrity} for more details on tensegrities of graphs).
\subsection{Toric variety associated to $G$.}\label{sec:graph}
Let $G$ be a simple graph without multiple edges and loops. We denote its edge set with $E(G)$ and its vertex set with $V(G)$. Fix a map $\mathbb{P}:G\rightarrow\RR^2$ such that $\mathbb{P}(V_i)\not=\mathbb{P}(V_j)$ for any pair of distinct vertices $V_i,V_j\in V(G)$. Assume that the images of edges are linear segments without interior intersections. Denote the image of $G$ by $\mathbb{P}(G)$. We require that $\mathbb{P}(G)$ is \emph{integral}, which means that the point $p_i=\mathbb{P}(V_i)$ has integral coordinates and the edge $\mathbb{P}(e)$ has a rational slope (or $\infty$), for every $V_i\in V(G)$ and $e\in E(G)$. Under this assumption, the balancing condition at $p_i$ has the form:
$$
    \sum_{j\not=i}w_{ij}\vect{p_ip_j}^{\ \prim}=0,
$$
where $w_{ij}$ is the stress on the edge connecting $p_i,p_j$ and $\vect{p_ip_j}^{\ \prim}$ is the primitive vector on the ray originating from $p_i$, and pointing to $p_j$. We may add extra edges on $\mathbb{P}(G)$ to obtain a triangulation $\mathcal{T}$ of the convex hull of $\mathbb{P}(G)$. We denote the boundary cycle of the convex hull of $\mathbb{P}(G)$ by $C$.

\medskip
Let $N\cong\ZZ^2$ be the lattice in $\RR^2$, and $\widetilde{N}=N\oplus\ZZ$. Put $\mathbb{P}(G)\subset \widetilde{N}_\RR$ at $(N_\RR,1)$. By taking cones over the triangulation $\mathcal{T}$, one can obtain a non-complete fan. (Note that we take one cone for every triangle in $\mathcal{T}$ and glue them together to obtain a polyhedral fan). We also add a new ray $\rho_0$ generated by the primitive vector $v_0$ along $-\sum_{i=1}^n v_i$ where $v_i=(p_i,1)$. Then, by including the cones $\sigma_{0ij}=\Cone\{v_0,v_i,v_j\}$ for every edge $\{v_i,v_j\}$ of the boundary cycle $C$, we obtain a complete polyhedral fan $\Sigma_{\mathcal{T}}\subset\widetilde{N}_\RR$, together with a complete toric variety $X_{\Sigma_\mathcal{T}}$.
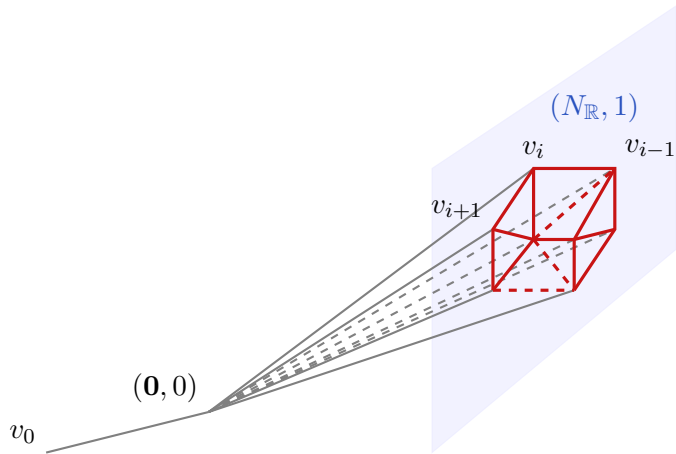
\begin{figure}[h!]
\begin{tikzpicture}[scale=0.27]
    \draw [gray, thick] (0,0) -- (-8,-2);
    \filldraw[blue!10, opacity=0.5] (11,-2) -- (11,12) -- (23,20) -- (23,8) -- cycle;
    \draw [gray, dashed, thick] (0,0) -- (18,8.5);
    \draw [gray, dashed, thick] (0,0) -- (16,8.5);
    \draw [gray, thick] (0,0) -- (14,6);
    \draw [gray, thick] (0,0) -- (14,9);
    \draw [gray, thick] (0,0) -- (16,12);
    \draw [gray, dashed, thick] (0,0) -- (20,12);
    \draw [gray, dashed, thick] (0,0) -- (20,9);
    \draw [gray, thick] (0,0) -- (18,6);
    \draw [very thick, DarkRed] (14,6) -- (14,9) -- (16,12) -- (20,12) -- (20,9) -- (18,6);
    \draw [very thick, dashed, DarkRed] (14,6) -- (18,6);
    \draw [very thick, DarkRed] (14,6) -- (16,8.5);
    \draw [very thick, DarkRed] (14,9) -- (16,8.5);
    \draw [very thick, DarkRed] (16,12) -- (16,8.5);
    \draw [very thick, DarkRed] (20,12) -- (18,8.5);
    \draw [very thick, DarkRed] (20,9) -- (18,8.5);
    \draw [very thick, DarkRed] (18,6) -- (18,8.5);
    \draw [very thick, DarkRed] (16,8.5) -- (18,8.5);
    \draw [very thick, dashed, DarkRed] (16,8.5) -- (20,12);
    \draw [very thick, dashed, DarkRed] (16,8.5) -- (18,6);
    \node [above left] at (0,0) {$({\bf 0},0)$};
    \node [above left] at (14,9) {$v_{i+1}$};
    \node [above] at (16,12) {$v_i$};
    \node [above right] at (20,12) {$v_{i-1}$};
    \node [above left] at (-8,-2) {$v_0$};
    \node [SkyBlue] at (19,15) {$(N_\RR,1)$};
\end{tikzpicture}
\caption{The polyhedral fan associated to $G$. The solid edges (in red) are the edges of $\mathbb{P}(G)$ and the dashed edges are the new edges to obtain a triangulation $\mathcal{T}$.}
    \label{toricconstruction}
\end{figure}
Let $E^+_\mathcal{T}$ be the set of union of the new edges in $\mathcal{T}$ but not in $\mathbb{P}(G)$ and define the set:
$$
    A^1(X_{\Sigma_\mathcal{T}};\QQ)_{\mathbb{P}(G)}:=\{c\in A^1(X_\Sigma;\QQ)\mid c(\tau)=0,\text{ for any } \tau\text{ over }e\in E_{\mathcal{T}}^+\}.
$$
\begin{proposition}\label{bc}
    $A^1(X_{\Sigma_\mathcal{T}};\QQ)_{\mathbb{P}(G)}$ is isomorphic to $A_\mathcal{F}$.
\end{proposition}
\begin{proof}
    We need to show that the balancing condition on the framework $\mathbb{P}(G)$ and the one on $\Sigma^1_\mathcal{T}$ are equivalent. At $p_i$, assume that the balancing condition on the $\ZZ$-framework is $\sum_{j\not=i}w_{ij}\vect{p_ip_j}^{\ \prim}=0$. For $\tau_{ij}=\Cone\{v_i,v_j\}$, let $c(\tau_{ij})=w_{ij}$. Note that $n_{\tau_{ij},\rho_i}=v_i+\vect{p_ip_j}^{\ \prim}$. Therefore, for any $\widetilde{m}=(m,t)\in \widetilde{M}(\rho_i)$, i.e.~$\langle m,p_i\rangle+t=0$, one has
    \begin{align*}
        \langle \widetilde{m},n_{\tau_{ij},\rho_i}\rangle&=\langle(m,t),(p_i+\vect{p_ip_j}^{\ \prim},1)\rangle\\
        &=\langle m,p_i\rangle+\langle m,\vect{p_ip_j}^{\ \prim}\rangle+t\\
        &=\langle m,\vect{p_ip_j}^{\ \prim}\rangle.
 \end{align*}
This shows that the balancing conditions on  $\mathbb{P}(G)$ and $\Sigma^1_\mathcal{T}$ are equivalent, as desired.\end{proof}

\begin{corollary}
   The framework $\mathbb{P}(G)$ admits a $\QQ$-tensegrity if and only if $A^1(X_\Sigma;\QQ)_{\mathbb{P}(G)}\not=0$.
\end{corollary}

\begin{remark}
 We note that in \cite{karpenkov2021frame}, Karpenkov provided some other criteria for the existence of tensegrities (see \cite[Theorems~2.18 and  5.20]{karpenkov2021frame}). We also note that for the construction of the (irreducible) toric variety, having a triangulation is not necessary. We have added this assumption, as the computation of Chow groups of cocycles can be easily carried out for the $\QQ$-factorial toric varieties. Moreover, adding the assistant ray $\rho_0$ is because we require a complete polyhedral fan to be able to apply the results from the intersection theory, see~\Cref{dual}.
\end{remark}

\subsection{Computing Chow rings.}\label{sec:computation} We now briefly review the computational method of Chow rings via Stanley-Reisner ideals. Let $z_i$ be the free generator corresponding to primitive vector $p_i$, for $i=1,\ldots,|\Sigma_1|$.
Since $\Sigma$ is simplicial, the Chow ring can be computed via
$$
    A^\bullet(X_\Sigma;\QQ)=\frac{\QQ[z_1,\ldots,z_{|\Sigma_1|}]}{\mathcal{SR}+\mathcal{LR}}
$$
where $\mathcal{SR}$ is the corresponding  Stanley-Reisner ideal and $\mathcal{LR}$ is the ideal of linear relations given by the lattice of characters. More explicitly, the Stanley-Reisner ideal $\mathcal{SR}$ is generated by $\{z_{i_1}\cdots z_{i_s}|\Cone\{v_{i_1},\cdots,v_{i_s}\}\not\in\Sigma\}$. Hence, the elements of $\mathcal{SR}$ have degree at least 2, and so
$$
    A^1(X_\Sigma;\QQ)=\frac{\Span_\QQ\{z_1,\ldots,z_{|\Sigma_1|}\}}{\mathcal{LR}}.
$$
In our case, we have that:
\[\mathcal{LR}=\langle\sum _{i=1}^n a_i z_i-\sum_{i=1}^na_iz_0,\sum _{i=1}^n b_iz_i-\sum _{i=1}^n b_iz_0,\sum _{i=1}^n z_i-nz_0\rangle.
\]
      
\medskip

We also note that the condition $A^1(X_\Sigma;\QQ)_{\mathbb{P}(G)}\not=0$ can be directly computed, using the formula in \cite[Proposition 6.4.4]{cox2011toric}. More explicitly,
let $\mult(\sigma)=[N_\sigma:\ZZ v_1+\cdots+\ZZ v_l]$, where $v_i$'s run over all rays of $\sigma$. 
Let $v_{i-1},v_{i+1}$ be the two neighbors of $v_i$. For $\tau_{0i}=\Cone\{v_0,v_i\},\sigma_{0,i,i\pm1}=\Cone\{v_0,v_i,v_{i\pm 1}\}$, one can compute the intersection number by the following formula:
\[
    D_j\cdot V(\tau_{0i})=\begin{cases}
        \ 0 \quad\quad\quad\quad\quad\  \quad\quad\quad j\not\in\{0,i,i\pm 1\}\\
        \displaystyle\frac{\mult(\tau_{0i})}{\mult(\sigma_{0,i,i\pm 1})}\quad\quad\quad j=i\pm 1\\
        \displaystyle\frac{\lambda_0\mult(\tau_{0i})}{\alpha\mult(\sigma_{0,i,i- 1})}\quad\quad j=0\\
        \displaystyle\frac{\lambda_i\mult(\tau_{0i})}{\alpha\mult(\sigma_{0,i,i-1})}\quad\quad j=i
    \end{cases}
\]
where $\alpha,\lambda_0,\lambda_i$ are determined by the linear relation $\alpha v_{i-1}+\lambda_0v_0+\lambda_iv_i+\beta v_{i+1}=0$.

\begin{figure}[h]
\centering
    \begin{subfigure}{0.45\linewidth}
    \centering
 \begin{tikzpicture}[scale=0.4]
    \draw [SkyBlue, dashed, very thick] (0,0) -- (4,0);
    \draw [SkyBlue, dashed, very thick] (0,0) -- (4,2);
    \draw [SkyBlue, dashed, very thick] (0,0) -- (5,1);
    \draw [DarkRed, very thick] (4,0) -- (4,2) -- (6,3) -- (7,0) -- cycle;
    \draw [DarkRed, very thick] (5,1) -- (4,0);
    \draw [DarkRed, very thick] (5,1) -- (4,2);
    \draw [DarkRed, very thick] (5,1) -- (6,3);
    \draw [DarkRed, very thick] (5,1) -- (7,0);
    \draw [SkyBlue, dashed, very thick] (6,3) -- (12,6);
    \draw [SkyBlue, dashed, very thick] (5,1) -- (10,2);
    \draw [SkyBlue, dashed, very thick] (7,0) -- (14,0);
    \draw [SkyBlue, dashed, very thick] (12,6) -- (10,2) -- (14,0) -- cycle;
    \node [below] at (5,1) {$p_5$};
    \node [below right] at (7,0) {$p_4$};
    \node [below left] at (4,0) {$p_3$};
    \node [above left] at (4,2) {$p_2$};
    \node [above right] at (6,3) {$p_1$};
    \node [right] at (12,6) {$p_1'$};
    \node [right] at (14,0) {$p_4'$};
    \node [below] at (10,2) {$p_5'$};
\end{tikzpicture}
\caption{$A^1_{\mathbb{P}(G)}(X_\Sigma;\QQ)\cong\QQ\not=0$}
\label{subfig1}
    \end{subfigure}
\begin{subfigure}{0.45\linewidth}
    \centering
 \begin{tikzpicture}[scale=0.7]
    \draw [SkyBlue, dashed, very thick] (-2,0) -- (4,0);
    \draw [SkyBlue, dashed, very thick] (1,1) -- (5,1);
    \draw [SkyBlue, dashed, very thick] (-2,0) -- (4,2);
    \draw [DarkRed, dashed, very thick] (5,1) -- (7,3);
    \draw [DarkRed, dashed, very thick] (5,1) -- (7,0);
    \draw [DarkRed, very thick] (4,0) -- (4,2) -- (7,3) -- (7,0) -- cycle;
    \draw [DarkRed, very thick] (5,1) -- (6,1);
    \draw [DarkRed, very thick] (5,1) -- (4,0);
    \draw [DarkRed, very thick] (5,1) -- (4,2);
    \draw [DarkRed, very thick] (6,1) -- (7,3);
    \draw [DarkRed, very thick] (6,1) -- (7,0);
    \node [above] at (6,1) {$p_6$};
    \node [above] at (5,1) {$p_5$};
    \node [below right] at (7,0) {$p_4$};
    \node [below left] at (4,0) {$p_3$};
    \node [above left] at (4,2) {$p_2$};
    \node [above right] at (6,3) {$p_1$};
\end{tikzpicture}
    \caption{$A^1_{\mathbb{P}(G)}(X_\Sigma;\QQ)=0$}
     \label{subfig2}
\end{subfigure}
    \caption{The computational examples in Example~\ref{ex:running}.}
    \label{fig.example}
\end{figure}
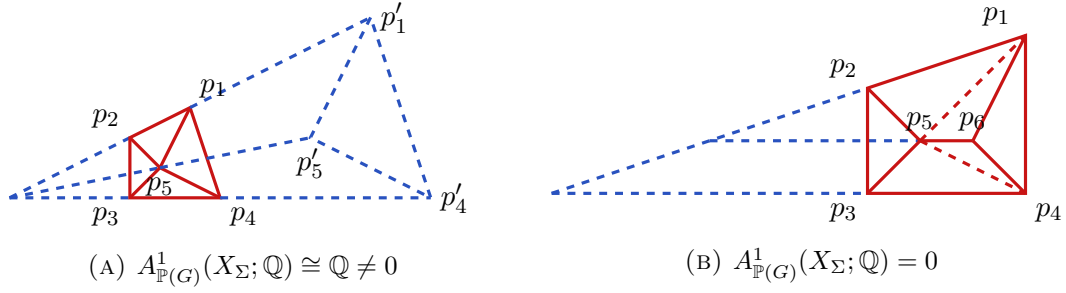

\begin{example}\label{ex:running}

    (a) Consider the $\ZZ$-framework with
$$
    p_1=(1,2),p_2=(-1,1),p_3=(-1,-1),p_4=(2,-1),p_5=(0,0),
$$
as shown in \Cref{fig.example}(A). Then $v_0=(-1,-1,-5)$, and the linear relation is
    \[
    \begin{cases}
        0=-z_0+z_1-z_2-z_3-2z_4\\
        0=-z_0+2z_1+z_2-z_3-z_4\\
        0=-5z_0+z_1+z_2+z_3+z_4+z_5
    \end{cases}
\]
So $A^1(X_\Sigma;\QQ)=\Span_\QQ\{D_0,D_1,D_2\}$.
The multiplicities of walls and full-dimensional cones are:
\smallskip

\begin{center}
\begin{tabular}{|c c c c|}
\hline
    $\mult(\tau_{01})=1$ &$\mult(\tau_{02})=2$ &$\mult(\tau_{03})=6$ &$\mult(\tau_{04})=3$\\

    $\mult(\tau_{12})=1$ &$\mult(\tau_{23})=2$ &$\mult(\tau_{34})=3$ &$\mult(\tau_{41})=1$\\

    $\mult(\tau_{15})=1$ &$\mult(\tau_{25})=1$ &$\mult(\tau_{35})=1$ &$\mult(\tau_{45})=1$\\

    $\mult(\sigma_{012})=14$ &$\mult(\sigma_{023})=12$ &$\mult(\sigma_{034})=18$ &$\mult(\sigma_{041})=21$\\

    $\mult(\sigma_{125})=3$ &$\mult(\sigma_{235})=2$ &$\mult(\sigma_{345})=3$ &$\mult(\sigma_{415})=5$\\
\hline
\end{tabular}
\end{center}
\smallskip

The intersection numbers are:
\smallskip

\begin{center}
\scalebox{0.85}{
\begin{tabular}{|c|c c c c c c c c c c c c|}
\hline
        &$V(\tau_{01})$ &$V(\tau_{02})$ &$V(\tau_{03})$ &$V(\tau_{04})$ &$V(\tau_{12})$ &$V(\tau_{23})$ &$V(\tau_{34})$ &$V(\tau_{41})$ &$V(\tau_{15})$ &$V(\tau_{25})$ &$V(\tau_{35})$ &$V(\tau_{45})$\\
\hline
    $D_0$ &1/42 &1/21 &1/6 &1/14 &1/14 &1/6 &1/6 &1/21 &0 &0 &0 &0\\
\hline
    $D_1$ &0 &1/7 &0 &1/7 &1/21 &0 &0 &1/35 &-1/15 &1/3&0&1/5\\
\hline
    $D_2$  &1/14 &-1/14 &1/2 &0 &-1/42 &0 &0 &0 &1/3 &-1/6 &1/2&0\\
\hline
    $D_3$ &0 &1/6 &0 &1/6 &0 &-1/6 & -1/6 &0 &0 &1/2&1/6 &1/3\\
\hline
    $D_4$ &1/21 &0 &1/3 &1/21 &0 &0 &0 &1/105 &1/5 &0 &1/3 &1/15\\
\hline
    $D_5$ &0 &0 &0 &0 &1/3 &1 &1 &1/5& -7/15 &-2/3 &-1 &-3/5\\
\hline
\end{tabular}
}
\end{center}

\medskip

Let $D=c_0D_0+c_1D_1+c_2D_2$, then the linear equation system $\{D\cdot V(\tau_{0j})=0,j=1,\ldots,4\}$ has nonzero solutions, which is  $A^1(X_\Sigma;\QQ)_{\mathbb{P}(G)}=\{\lambda(-6D_0+3D_1+2D_2),\lambda\in\QQ$\}. 

\medskip

(b) For the $\ZZ$-framework with
$$
    p_1=(2,2),p_2=(-1,1),p_3=(-1,-1),p_4=(2,-1),p_5=(0,0),p_6=(1,0)
$$
and edges as in Figure~\ref{fig.example}(B), we have that $A^1(X_\Sigma;\QQ)_{\mathbb{P}(G)}=0$. For example, let $D=\sum_{i=0}^3c_iD_i$. Then $\{D\cdot V(\tau_{0j})=0,j=1,\ldots,4\}$ has solutions $\{D=\lambda(-7D_0+3D_1+2D_2),\lambda\in\QQ\}$. However, $D\cdot V(\tau_{15})=0$ forces $\lambda=0$.

\medskip

We note that the computational results in this example agree with the conclusions in \cite[Example 1.2]{karpenkov2021frame} and \cite[Proposition 5.11]{karpenkov2021frame}.
\end{example}

\begin{remark}
    In general, in the graph $\mathbb{P}(G)$, the edges may intersect in points other than vertices. We note that, to deal with this problem, instead of using classical toric varieties, the theory of multi-fans in \cite{hattori2003theory,Masuda} can be applied. Moreover, to work over $\RR$ instead of $\QQ$ or $\ZZ$, the techniques developed in \cite{battaglia2001generalized} are relevant.
\end{remark}

\bibliographystyle{alpha}
\bibliography{reference}
\smallskip

\noindent 
\footnotesize{\textbf{Authors' addresses:}
\smallskip

\noindent
(Mohammadi) Department of Computer Science, KU Leuven, Celestijnenlaan 200A, B-3001 Leuven, Belgium\\ 
   Department of Mathematics, KU Leuven, Celestijnenlaan 200B, B-3001 Leuven, Belgium 
\\ Department of Mathematics and Statistics,
 UiT – The Arctic University of Norway, 9037 Troms\o, Norway
 \\ E-mail address: {\tt fatemeh.mohammadi@kuleuven.be}
\\
\noindent (Wu) E-mail address: {\tt xianwu.ag@gmail.com}
}
\end{document}